\documentclass{amsart}
\usepackage{amsfonts}
\usepackage{amsmath,amssymb}
\usepackage{amsthm}
\usepackage{amscd}
\usepackage{graphics}
\usepackage{graphicx}

\theoremstyle{definition}{
\newtheorem{Def}{{\rm Definition}}

\newtheorem{Rem}{{\rm Remark}}

}
\theoremstyle{plane}
{
\newtheorem{Cor}{Corollary}

\newtheorem{Thm}{Theorem}
\newtheorem*{MainThm}{Main Theorem}

}

\begin{document}
\title[Fold maps on simply-connected manifolds cohomologically ${\mathbb{C}P}^2 \times S^3$]{$7$-dimensional simply-connected spin manifolds whose integral cohomology rings are isomorphic to that of ${{\mathbb{C}}P}^2 \times S^3$ admit round fold maps}
\author{Naoki Kitazawa}
\address{Institute of Mathematics for Industry, Kyushu University, 744 Motooka, Nishi-ku Fukuoka 819-0395, Japan\\
 TEL (Office): +81-92-802-4402 \\
 FAX (Office): +81-92-802-4405 \\
}
\email{n-kitazawa@imi.kyushu-u.ac.jp}
\urladdr{https://naokikitazawa.github.io/NaokiKitazawa.html}
\maketitle
\begin{abstract}

We have been interested in understanding the class of $7$-dimensional closed and simply-connected manifolds in geometric and constructive ways. We have constructed explicit {\it fold} maps, which are higher dimensional versions of Morse functions, on some of the manifolds, previously.

The studies have been motivated by studies of {\it special generic} maps, higher dimensional versions of Morse functions on homotopy spheres with exactly two singular points, characterizing them topologically except $4$-dimensional cases. The class contains canonical projections of unit spheres for example.

 This class has been found to be interesting, restricting the topologies and the differentiable structures of the manifolds strictly: Saeki, Sakuma and Wrazidlo found explicit phenomena.
The present paper concerns fold maps on $7$-dimensional closed and simply-connected spin manifolds whose integral cohomology rings are isomorphic to that of the product of the $2$-dimensional complex ptojective space and the $3$-dimensional sphere.



\end{abstract}


\maketitle
\section{Introduction, terminologies and notation.}
\label{sec:1}

$7$-dimensional closed and simply-connected manifolds are important objects in the theory of classical algebraic topology and differential topology (of higher dimensional closed and simply-connected manifolds). The class has been attractive since the discoveries of $7$-dimensional exotic homotopy spheres by Milnor \cite{milnor}.
 Still recently, new understandings via algebraic topological tools such as characteristic classes and bordism relations have been studied by Kreck \cite{kreck} and Wang \cite{wang} for example.

We introduce some terminologies, notions and notation on smooth manifolds.
${\mathbb{R}}^k$ denotes the $k$-dimensional Euclidean space, endowed with the standard Euclidean metric for each integer $k \geq 0$. $||x|| \geq 0$ denotes the distance between $x \in {\mathbb{R}}^k$ and the origin $0$. $S^k \subset {\mathbb{R}}^{k+1}$ denotes the set of all points such that the distances between the points and the origin are $1$ for $k\geq 0$. This is the $k$-dimensional unit sphere. A {\it homotopy sphere} means a smooth manifold homeomorphic to a unit sphere and it is said to be  {\it standard} (resp. {\it exotic}) if it is (resp. not) diffeomorphic to a (resp. any) unit sphere. $D^k \subset {\mathbb{R}}^k$ denotes the set of all points such that the distances between the points and the origin are smaller than or equal to $1$ for $k \geq 0$. This is the $k$-dimensional unit disk. A {\it standard} disk means a smooth manifold diffeomorphic to a unit disk. ${\mathbb{C}P}^k$ denotes the $k$-dimensional complex projective space.
We consider connected sums and boundary connected sums of manifolds in the smooth category, unless otherwise stated.


The author has been interested in understanding the class of the $7$-dimensional closed and simply-connected manifolds
 in geometric and constructive ways and obtained explicit {\it fold} maps, which are higher dimensional versions of Morse functions, on the manifolds.
The studies have been motivated by studies of {\it special generic} maps, which are higher dimensional versions of Morse functions on homotopy spheres with exactly two singular points. It is well-known that a closed manifold is homeomorphic to a sphere if and only if it admits such a function except $4$-dimensional cases.
 In $4$-dimensional cases, a closed manifold is diffeomorphic to the standard sphere if and only if it admits such a function. 
The class of special generic maps contains canonical projections of unit spheres for example. This class has been found to be interesting, restricting the topologies and the differentiable structures of homotopy spheres and manifolds admitting them strictly. For example $7$-dimensional exotic homotopy spheres admit no special generic maps into ${\mathbb{R}}^n$ for $n=4,5,6$.
 In considerable cases they admit no special generic maps for $n=3$. See \cite{saeki,saeki2, saekisakuma,saekisakuma2,wrazidlo}.

It is also an important fact that construction of explicit fold maps on explicit manifolds is difficult even for fundamental manifolds and in general. 
The present paper concerns fold maps on $7$-dimensional simply-connected spin manifolds whose integral cohomology rings are isomorphic to that of ${{\mathbb{C}}P}^2 \times S^3$. Related expositions on the existence of fold maps and explicit fold maps are in section \ref{sec:3} or appendices.

A recent study on classifications of these manifolds by \cite{wang} also motivates the author to do this study.
We introduce a theorem on the classifications. We omit precise expositions on {\it spin} manifolds and ({\it $j$-th}) {\it Pontryagin classes} for example.

Hereafter, the $j$-th integral (co)homology group of a topological space $X$ is denoted by $H_j(X;\mathbb{Z})$ (resp. $H^j(X;\mathbb{Z})$). The integral cohomology ring is denoted by $H^{\ast}(X;\mathbb{Z})$.

\begin{Thm}[\cite{wang}]
\label{thm:1}
\begin{enumerate}
\item For an arbitrary $7$-dimensional, closed, simply-connected and spin manifold $X$ whose integral cohomology ring is isomorphic to that of ${\mathbb{C}P}^2 \times S^3$, consider an isomorphism $\phi$ of integral cohomology rings from the integral cohomology ring of ${\mathbb{C}P}^2 \times S^3$ onto that of $X$ and set $a$ as a generator of $H^2({\mathbb{C}P}^2 \times S^3;\mathbb{Z}) \cong \mathbb{Z}$.

 In this situation, the 1st Pontryagin class of $X$ is represented as $4k\phi(a)$ for some integer $k$. Moreover, the topology of a $7$-dimensional closed, simply-connected and spin manifold whose integral cohomology ring is isomorphic to that of ${\mathbb{C}P}^2 \times S^3$ is determined by its 1st Pontryagin class.
\item For any pair $(M_1,M_2)$ of mutually homeomorphic $7$-dimensional, closed and simply-connected spin manifolds whose integral cohomology rings are isomorphic to that of ${\mathbb{C}P}^2 \times S^3$, $M_1$ is represented as a connected sum of $M_2$ and a suitable $7$-dimensional homotopy sphere $\Sigma$. 

Moreover, for a $7$-dimensional, closed and simply-connected spin manifold $M_0$ whose integral cohomology ring is isomorphic to the integral cohomology ring of ${\mathbb{C}P}^2 \times S^3$, consider a pair $(M_1,M_2)$ of manifolds such that $M_i$ is represented as a connected sum of $M_0$ and a homotopy sphere ${\Sigma}_i$. $M_1$ and $M_2$ are diffeomorphic if and only if ${\Sigma}_1$ and ${\Sigma}_2$ are mutually diffeomorphic.
\end{enumerate} 
\end{Thm}

\cite{kreck} gives a classification of $7$-dimensional closed and simply-connected (spin) manifolds whose 2nd integral homology groups are free via concrete bordism theory.
The integral homology group $H_j({{\mathbb{C}}P}^2 \times S^3;\mathbb{Z})$ is $\{0\}$ for $j=1,6$ and isomorphic to $\mathbb{Z}$ for $j=2,3,4,5$. As a fundamental  important fact, its integral cohomology ring is not isomorphic to that of any $7$-dimensional closed and simply-connected manifold represented as a connected sum of products of spheres. If for a $7$-dimensional closed and simply-connected manifold whose $j$-th integral homology group is free for any $j$, the rank of the $j_0$-th integral homology group is at most one for $j_0=2,3,4,5$ and at least one of these for groups is of rank $0$, then the integral cohomology ring is isomorphic to that of some manifold represented as a connected sum of products of spheres.

Note also that neither ${\mathbb{C}P}^2$ nor ${\mathbb{C}P}^2 \times S^k$ is spin for $k \geq 1$.

\subsection{Terminologies and notation on smooth manifolds, maps and bundles.}
\label{subsec:1.1}


Throughout the present paper, manifolds and maps between manifolds are smooth (of class $C^{\infty}$) unless otherwise stated. Diffeomorphisms on manifolds are always smooth and the {\it diffeomorphism group} of a manifold is defined as the group of all smooth diffeomorphisms on the manifold. 
For bundles whose fibers are manifolds, the structure groups are subgroups of the diffeomorphism groups in the present paper. In other words, the bundles are assumed to be {\it smooth}. A {\it linear} bundle means a smooth bundle whose fiber is a Euclidean space, a unit sphere or a unit disk with structure group acting linearly in a canonical way on the fiber.
 For general theory of linear bundles and their characteristic classes such as {\it Stiefel-Whitney classes} and {\it Pontryagin classes} and more general bundles, see \cite{milnorstasheff, steenrod}. We consider these notions and related ones on bundles in the present paper. We omit explaining about the definitions and fundamental properties on them in the present paper.
We also omit precise expositions on characteristic classes of smooth manifolds defined as those of the tangent bundles of smooth manifolds.


A {\it singular} point of a smooth map is a point in the manifold of the domain at which the dimension of the image of the differential is smaller than both the dimensions of the manifolds of the domain and the target.
 We call the set of all singular points the {\it singular set} of the map. We call the image of the singular set the {\it singular value set} of the map. The {\it regular value set} of the map is the complementary set of the singular value set. A {\it singular {\rm (}regular{\rm )} value} is a point in the singular (resp. regular) value set. For a smooth map $c$, $S(c)$ denotes the singular set of $c$.


\subsection{The definition and fundamental properties of a fold map, fold maps on $7$-dimensional homotopy spheres, round fold maps and Main Theorem.}
\label{subsec:1.2}
Hereafter, let $m \geq n \geq 1$ be integers.
 A smooth map from an $m$-dimensional smooth manifold with no boundary into an $n$-dimensional smooth manifold with no boundary is said to be a {\it fold} map if at each singular point $p$, the map is represented as $(x_1, \cdots, x_m) \mapsto (x_1,\cdots,x_{n-1},\sum_{k=n}^{m-i}{x_k}^2-\sum_{k=m-i+1}^{m}{x_k}^2)$ for suitable coordinates and an integer $0 \leq i(p) \leq \frac{m-n+1}{2}$. For the singular point $p$, $i(p)$ is unique and called the {\it index} of $p$. The set consisting of all singular points of a fixed index of the map is a closed submanifold of dimension $n-1$ with no boundary of the $m$-dimensional manifold. The restriction map to the singular set is an immersion.
A {\it special generic} map, which is presented before, is defined as a fold map such that the index of a singular point is always $0$.  
We have obtained the following theorem before through challenging construction of explicit fold maps on explicit manifolds. For classical theory on $7$-dimensional homotopy spheres, see \cite{eellskuiper, milnor} for example.
\begin{Thm}[\cite{kitazawa, kitazawa2}]
\label{thm:2}
\begin{enumerate}
\item 
\label{thm:2.1}
Every $7$-dimensional oriented homotopy sphere $M$ of $28$ types admits a fold map $f:M \rightarrow {\mathbb{R}}^4$ satisfying the following properties.
\begin{enumerate}
\item $f {\mid}_{S(f)}$ is an embedding satisfying $f(S(f))=\{x \in {\mathbb{R}}^4 \mid ||x||=1,2,3\}$.
\item The index of each singular point is $0$ or $1$.
\item The preimage of a regular value in each connected component of $f$ is, empty, diffeomorphic to $S^3$, diffeomorphic to $S^3 \sqcup S^3$ and diffeomorphic to $S^3 \sqcup S^3 \sqcup S^3$, respectively.
\end{enumerate}
\item
\label{thm:2.2}
A $7$-dimensional homotopy sphere $M$ admits a fold map $f:M \rightarrow {\mathbb{R}}^4$ such that $f {\mid}_{S(f)}$ is an embedding and that $f(S(f))=\{x \in {\mathbb{R}}^4 \mid ||x||=1\}$ if and only if $M$ is a standard sphere.
A $7$-dimensional homotopy sphere $M$ admits a fold map $f:M \rightarrow {\mathbb{R}}^4$ such that $f {\mid}_{S(f)}$ is an embedding and that the index of each singular point is $0$ or $1$ as before and the following properties hold if and only if $M$ is the total space of a linear bundle whose fiber is $S^3$ over $S^4$ {\rm (}$7$-dimensional oriented homotopy spheres of $16$ types including the unit sphere are represented in this way{\rm )}.
\begin{enumerate}
\item
\label{thm:2.2.1}
 $f(S(f))=\{x \in {\mathbb{R}}^4 \mid ||x||=1,2\}$.
\item
\label{thm:2.2.2}
 For any connected component $C \subset f(S(f))$ and a small closed tubular neighborhood $N(C)$, the bundle given by the projection represented as the composition of $f {\mid}_{f^{-1}(N(C))}$ with the canonical projection to $C$ is trivial.
\item
\label{thm:2.2.3}
 The preimage of a regular value in each connected component is, empty, diffeomorphic to $S^3$, and diffeomorphic to $S^3 \sqcup S^3$, respectively.
\end{enumerate}
\end{enumerate}
\end{Thm}
Especially, $f$ is special generic if and only if $f(S(f))=\{x \in {\mathbb{R}}^4 \mid ||x||=1\}$. This explicitly presents that the numbers of connected components of the singular value sets affect the differentiable structures of $7$-dimensional homotopy spheres for suitable explicit classes of fold maps which may not be the one of special generic maps. 
In addition, the maps here belong to the class of {\it round} fold maps.
\begin{Def}[\cite{kitazawa,kitazawa2,kitazawa3,kitazawa4,kitazawa5}]
For an integer $n \geq 2$, a fold map $f$ into ${\mathbb{R}}^n$ is said to be {\it round} if $f {\mid}_{S(f)}$ is an embedding and for a suitable diffeomorphism $\phi$ on ${\mathbb{R}}^n$ and an integer $l>0$, $(\phi \circ f)(S(f))={\bigcup}_{r=1}^l \{||x||=r \mid x \in {\mathbb{R}}^n\}$.
\end{Def}
This notion is also defined for $n=1$ in \cite{kitazawa5}. However we do not need this in the present paper.
The following theorem is our main theorem. 
\begin{MainThm}
There exists an infinite family $\{M_k\}_{k \in \mathbb{Z}}$ of closed and simply-connected spin {\rm (}oriented{\rm )} manifolds whose integral cohomology rings are isomorphic to that of ${\mathbb{C}P}^2 \times S^3$ such that the 1st-Pontryagin class of $M_k$ is $4k$ times a generator of $H^4(M_k;\mathbb{Z}) \cong \mathbb{Z}$ and these manifolds admit round fold maps $\{f_k:M_k \rightarrow {\mathbb{R}}^4\}$. Furthermore, {\rm (}by virtue of Theorem \ref{thm:1}{\rm )}, every $7$-dimensional, closed, simply-connected and spin manifold whose integral cohomology ring is isomorphic to that of $\mathbb{C}{P}^2 \times S^3$ admits a round fold map into ${\mathbb{R}}^4$.
\end{MainThm}

\subsection{The content of the paper}
\label{subsec:1.3}

The next section is devoted to the proof of Main Theorem. 
We construct a round fold map on a $7$-dimensional, closed and simply-connected manifold, whose integral cohomology ring is isomorphic to that of ${\mathbb{C}}P^2 \times S^3$. 

Section \ref{sec:3} is for expositions on fold maps on $7$-dimensional closed and simply-connected (spin) manifolds. Here we give expositions on existence of fold maps on closed manifolds represented as connected sums of products of standard spheres and manifolds of some wider classes. We also explain about non-existence of fold maps on projective spaces in some general cases. We also introduce related recent results announced by the author, which do not affect our present arguments and new results.
\section{The proof of Main Theorem.}
We prove Main Theorem. 
We introduce fundamental terminologies and notions.


For a closed and orientable manifold, consider an integral homology class of degree $k$. It is said to be {\it represented} by a closed, connected and oriented submanifold of dimension $k$ with no boundary if the class is equal to the value of the homomorphism canonically induced from the inclusion at the {\it fundamental class} of the submanifold: the {\it fundamental class} is a generator of the integral cohomology of degree $k$ (of the $k$-dimensional closed and oriented manifold) compatible with the given orientation. 
For related notions on elementary algebraic topology, see \cite{hatcher} for example.
\subsection{The $3$-dimensional complex projective space and its structure.}
The $3$-dimensional complex projective space $\mathbb{C}{P}^3$ is regarded as the total space of a linear bundle over $S^4$ whose fiber is $S^2$. It is a spin manifold.
The 1st Pontryagin class of the projective space is $4$ times a generator of the 4th integral cohomology group.
 For this, see \cite{little} for example. For classical systematic theory of Pontryagin classes see \cite{milnorstasheff}.
The following theorem presents several classical and important properties.

\begin{Thm}
\label{thm:3}
For the projection ${\pi}_{{\mathbb{C}P}^3}:{\mathbb{C}P}^3 \rightarrow S^4$ of the linear bundle over $S^4$ whose fiber is $S^2$ and the total space ${\mathbb{C}P}^3$, the following properties hold.
\begin{enumerate}
\item There exists a complex projective plane ${{\mathbb{C}P}^2}_{{\mathbb{C}P}^3} \subset {\mathbb{C}P}^3$ being also a complex submanifold. Furthermore, a generator of $H_4({\mathbb{C}P}^3;\mathbb{Z})$, isomorphic to $\mathbb{Z}$, is represented by the submanifold with a suitable orientation.
\item ${\pi}_{{\mathbb{C}P}^3} {\mid}_{{{\mathbb{C}P}^2}_{{\mathbb{C}P}^3}}$ gives a diffeomorphism on the preimage of a smoothly embedded $4$-dimensional standard closed disk in $S^4$.
\item There exists a complex projective line ${{\mathbb{C}P}^1}_{{\mathbb{C}P}^2} \subset {{\mathbb{C}P}^2}_{{\mathbb{C}P}^3}$ being also a complex submanifold. Furthermore, a generator of $H_2({\mathbb{C}P}^3;\mathbb{Z})$, isomorphic to $\mathbb{Z}$, is represented by the submanifold where the submanifold is suitably oriented. The complex projective line can be taken as a fiber of the bundle.
\item The square of the dual of the integral homology class represented by ${{\mathbb{C}P}^1}_{{\mathbb{C}P}^2}$ is the dual of the integral homology class represented by ${{\mathbb{C}P}^2}_{{\mathbb{C}P}^3}$ with the given orientation. The product of the dual of the integral homology class represented by ${{\mathbb{C}P}^1}_{{\mathbb{C}P}^2}$ and the dual of the integral homology class represented by ${{\mathbb{C}P}^2}_{{\mathbb{C}P}^3}$ is a generator of $H^6({\mathbb{C}P}^3;\mathbb{Z})$, isomorphic to $\mathbb{Z}$.
\end{enumerate}
\end{Thm}
\subsection{A proof of Main Theorem.}
We consider a submersion obtained by composing the canonical projection from ${\mathbb{C}P}^3 \times S^1$ to ${\mathbb{C}P}^3$ with ${\pi}_{{\mathbb{C}P}^3}:{\mathbb{C}P}^3 \rightarrow S^4$ in Theorem \ref{thm:3}.
We have a $7$-dimensional, closed and simply-connected manifold $M$ and a fold map $f_{S^4}:M \rightarrow S^4$ by exchanging the restriction to the preimage of a smoothly embedded $4$-dimensional standard disk $D$ so that the following properties hold. 
Let $D^{\prime} \subset S^4$ be a smoothly embedded $4$-dimensional standard disk containing the disk $D$ in the interior.
\begin{itemize}
\item $f_{S^4}$ is a fold map satisfying the following properties.
\begin{itemize}
	\item The singular set $S(f_{S^4})$ is diffeomorphic to $S^3$.
	\item $f_{S^4} {\mid}_{S(f_{S^4})}$ is an embedding.
	\item $f_{S^4}(S(f_{S^4}))=\partial D^{\prime}$.
	\item The preimage of a small closed tubular neighborhood of $\partial D^{\prime}$ in $S^4$ is regarded as a trivial bundle. The structure of the bundle is given by the composition of the restriction of $f_{S^4}$ with the canonical projection to $\partial D^{\prime}$.
	\end{itemize}
\item The preimage of a regular value of each connected component of $S^4-f_{S^4}(S(f_{S^4}))$ is $S^3$ and $S^2 \times S^1$, respectively. 
\end{itemize}
Note that in the second property, (a copy of) $S^3$ is regarded as a manifold obtained by a so-called handle attachment to (a copy of) $S^2 \times S^1$. In other words, we attach a {\it 3-handle} $D^3 \times D^1$ along $S^2 \times \{\ast\} \subset S^2 \times S^1$, respecting the structure of a Morse function. See \cite{milnor2} for example for related theory.
The following show topological properties of $M$. 
Here we fix a suitable basis or a generating set for $H_j(M;\mathbb{Z})$ to define the duals of some homology classes. In fact it seems to be not so difficult to guess that $H_j(M;\mathbb{Z})$ is free and of rank $0$ or $1$ and we show this later.
We can see the properties by investigating the presented two properties of the map before.

\begin{itemize}
\item There exists a complex projective plane ${{\mathbb{C}P}^2}_{M} \subset M$ being also a submanifold such that ${f_{S^4}} {\mid}_{{{\mathbb{C}P}^2}_{M}}$ gives a diffeomorphism on the preimage of the disk $D$ in $S^4$.

\item The integral homology class represented by the submanifold $S^2 \times \{\ast\} \subset S^2 \times S^1$ of the preimage of a regular value diffeomorphic to $S^2 \times S^1$ is a generator of a subgroup of rank $1$ of $H_2(M;\mathbb{Z})$ where $S^2 \times \{\ast\}$ is oriented suitably.
The integral homology classes represented by the preimages of regular values are a generator of a subgroup of rank 1 of $H_3(M;\mathbb{Z})$. The square of the dual of the integral homology class represented by $S^2 \times \{\ast\}$ with the orientation before is the dual of the integral homology class represented by the complex projective plane ${{\mathbb{C}P}^2}_{M} \subset M$ with a suitable orientation.
\item $M$ is spin.
\item The 1st Pontryagin class of $M$, which is oriented suitably, is $4$ times the dual of the fundamental class of the complex projective plane ${{\mathbb{C}P}^2}_{M} \subset M$ with the orientation before. 
\end{itemize}

We take the union of two smoothly and disjointly embedded $4$-dimensional standard disks $D_1:=D$ and $D_2$ so that they are mutually in distinct connected components of the regular value set of $f_{S^4}$. 
By composing the restriction of $f_{S^4}$ to the preimage ${f_{S^4}}^{-1}(D_1 \sqcup D_2)$ with a natural $2$-fold covering over a $4$-dimensional standard disk $D_0$, we have a trivial smooth bundle over $D_0$ whose fiber is diffeomorphic to $S^3 \sqcup (S^2 \times S^1)$. We consider the projection and embed the base space as the space $\{x \in {\mathbb{R}}^4 \mid  |x| \leq r_0\}$ for a positive integer $r_0>0$. Let $\tilde{f_{r_0}}$ denote the resulting submersion.
The restriction of $f_{S^4}$ to ${f_{S^4}}^{-1}(S^4-{\rm Int}\ (D_1 \sqcup D_2))$ is regarded as the product map of a Morse function with exactly one singular point and the identity map on ${\rm id}_{S^3}$. 
On the manifold of the domain of this function, we can construct the product map of a Morse function $\tilde{f}$ and the identity map on ${\rm id}_{S^3}$. Furthermore, we can construct the Morse function $\tilde{f}$ satisfying the following properties.
\begin{itemize}
\item The function is a function on a manifold diffeomorphic to a manifold obtained by removing the interior of a $4$-dimensional standard disk smoothly embedded in the interior of $S^2 \times D^2$.
\item The minimum is $r_0$.
\item The preimage of the minimum $r_0$ of the function coincides with the boundary and contains no singular points.
\item The function has exactly three singular points and at distinct singular points, the values are distinct.
\end{itemize}
We can define the function $\tilde{f}$ as one to the half-closed interval $[r_0,+\infty) \subset \mathbb{R}$. Thanks to the structures of manifolds and maps, we can glue the maps $\tilde{f_{r_0}}$ and the product map of $\tilde{f}$ and the identity map on the boundary of $\{x \in {\mathbb{R}}^4 \mid  |x| \leq r_0\}$ to obtain a global fold map on the original manifold $M$ into ${\mathbb{R}}^4$.

Furthermore, the topological properties of $M$ and $f_{S^4}:M \rightarrow S^4$ presented before together with fundamental propositions on 1st Pontryagin classes enable us to construct a similar fold map on a manifold whose integral cohomology ring is isomorphic to that of $M$, which is oriented suitably, and whose 1st Pontryagin class is $4k$ times a generator of the 4th integral homology group, isomorphic to $\mathbb{Z}$, for an arbitrary integer $k$. We need to change the way we glue the trivial bundle over $D_0$ whose fiber is diffeomorphic to $S^3$. For Pontryagin classes see \cite{milnorstasheff} again. See also Theorems 3 and 4 of \cite{kitazawa6}.
We can see that $M$ is simply-connected and that the $j$-th integral homology group $H_j(M;\mathbb{Z})$ of $M$ is isomorphic to zero for $j=1,6$ and $\mathbb{Z}$ for $j=2,3,4,5$. For example, we can see that by composing a suitable diffeomorphism on ${\mathbb{R}}^4$ and a canonical projection to $\mathbb{R}$, we have a Morse function with exactly six singular points and this with the presented properties of $M$ shows this. We can see that the integral cohomology ring of $M$ is isomorphic to that of ${\mathbb{C}P}^2 \times S^3$. 
This yields our Main Theorem as follows.

\begin{Thm}
\label{thm:4}
There exist an infinite family $\{M_k\}_{k \in \mathbb{Z}}$ of closed and simply-connected spin {\rm (}oriented{\rm )} manifolds whose integral cohomology rings are isomorphic to that of ${\mathbb{C}P}^2 \times S^3$ such that the 1st-Pontryagin class of $M_k$ is $4k$ times a generator of $H^4(M_k;\mathbb{Z}) \cong \mathbb{Z}$ and a family of round fold maps $\{f_k:M_k \rightarrow {\mathbb{R}}^4\}$ satisfying the following properties.
\begin{enumerate}
\item The singular set of each map consists of exactly three connected components. More precisely, we can construct the maps such that the restrictions to the singular sets are embeddings and that the images are $\{x \in {\mathbb{R}}^4 \mid ||x||=1,2,3\}$.
\item For each map, the preimage of a regular value in each connected component of the regular value set in the image is diffeomorphic to $S^3$, $S^2 \times S^1$ and $S^3 \sqcup S^2 \times S^1$, respectively.
\end{enumerate}  
\end{Thm}

See also FIGURE \ref{fig:1}.

\begin{figure}

\includegraphics[width=40mm]{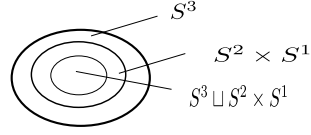}
\caption{The image and the preimages of regular values of a round fold map in Theorem \ref{thm:4}: circles represent the singular value set and $3$-dimensional spheres.}
\label{fig:1}
\end{figure}

Theorem \ref{thm:1} yields the following corollary.
\begin{Cor}
\label{cor:1}

A $7$-dimensional, closed and simply-connected spin manifold whose integral cohomology ring is isomorphic to that of ${\mathbb{C}P}^2 \times S^3$ always admits a round fold map into ${\mathbb{R}}^4$.
\end{Cor}
\begin{proof}
Theorem \ref{thm:1} implies that a $7$-dimensional, closed, simply-connected and spin manifold whose integral cohomology ring is isomorphic to that of ${\mathbb{C}P}^2 \times S^3$ is always represented as a connected sum of a manifold in Theorem \ref{thm:4} and a $7$-dimensional homotopy sphere. \cite{kitazawa, kitazawa2, kitazawa5}
 show construction of a new round fold map from a round fold map in Theorem \ref{thm:4} and a round fold map in Theorem \ref{thm:2}. We take a $4$-dimensional standard disk in the connected component of the regular value set of the former map in the center and remove the interior of a connected component, regarded as the total space of a trivial smooth bundle over the disk whose fiber is $S^3$ of the preimage. We remove (the preimage of) the union of the interior of a small closed tubular neighborhood of the outermost connected component of the singular value set of the latter map and the complementary set of the image. We glue the two obtained maps in a suitable way to obtain a round fold map on the desired manifold. 

We give a more precise exposition. Let $M_1$ denote the manifold of the domain of the former map. Let $M_2$ denote the homotopy sphere of the domain of the latter map. The preimage of the interior of a small closed tubular neighborhood of the outermost connected component of the singular value set of the latter map is regarded as a small closed tubular neighborhood $N_0(S)$ of a smoothly embedded $3$-dimensional standard sphere in $M_2$. The closed tubular neighborhood is a trivial linear bundle whose fiber is diffeomorphic to $D^4$ and it can be regarded as a subbundle of a normal bundle of a sphere in the interior of a $7$-dimensional standard disk smoothly embedded in $M_2$ after we move the $3$-dimensional sphere via a suitable smooth isotopy. The $3$-dimensional sphere is embedded there as a so-called {\it unknot} in the smooth category. $M_2-{\rm Int}\ N_0(S)$ is diffeomorphic to $D^4 \times S^3$ and regarded as the total space of a trivial smooth bundle over $D^4$ whose fiber is $S^3$. 

From $M_1$, the total space of a trivial smooth bundle over $D^4$ whose fiber is diffeomorphic to $S^3$ is removed as before. Here $D^4$ is regarded as a smoothly embedded disk in the connected component of the regular value in the center. By gluing the resulting compact and connected manifolds obtained from $M_1$ and $M_2$ by a suitable isomorphism between the trivial smooth bundles defined canonically on the boundaries, we have a new manifold $M$ and a round fold map $f:M \rightarrow {\mathbb{R}}^4$.


Furthermore, we can have the manifold $M$ which is represented as a connected sum of the original two manifolds. 

We also give an additional exposition on round fold maps in Theorem \ref{thm:2} (\ref{thm:2.2}). We can also have a fold map on every $7$-dimensional homotopy sphere in (\ref{thm:2.1}) from two maps in (\ref{thm:2.2}) in the presented way. In fact every $7$-dimensional homotopy sphere is represented as a connected sum of two oriented homotopy spheres of the 16 types in (\ref{thm:2.2}) according to the well-known theory of $7$-dimensional homotopy spheres.

For the construction here, see also FIGUREs \ref{fig:2}, \ref{fig:3} and \ref{fig:4}. \cite{kobayashisaeki} has also motivated the author to find the methods in \cite{kitazawa, kitazawa2, kitazawa5} and here. 

This completes the proof. 
\end{proof}

\begin{figure}

\includegraphics[width=70mm]{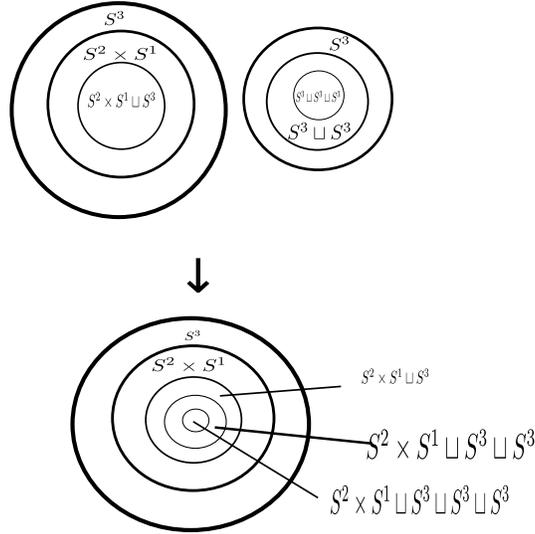}
\caption{Construction of a round fold map on the manifold represented as a connected sum of the two manifolds admitting fold maps in Corollary \ref{cor:1}: for example manifolds represent preimages of regular values in the connected components of the regular value sets. More precisely, we attach maps presented in FIGURE \ref{fig:3} and FIGURE \ref{fig:4} to obtain the desired map.}
\label{fig:2}
\end{figure}
\begin{figure}

\includegraphics[width=25mm]{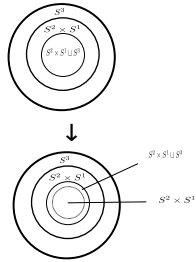}
\caption{The image of a round fold map in Theorem \ref{thm:4}. The image of the map obtained by removing a connected component of the preimage of the interior of a $4$-dimensional standard disk embedded smoothly in the innermost connected component of the regular value set.}
\label{fig:3}
\end{figure}
\begin{figure}

\includegraphics[width=25mm]{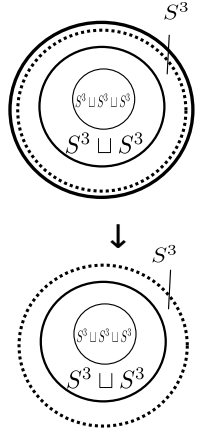}
\caption{The image of a round fold map in Theorem \ref{thm:2}. The image of the map obtained by removing the union of the complementary set of the image and the interior of a small closed tubular neighborhood of the outermost connected component of the singular value set (and the preimage).}
\label{fig:4}
\end{figure}
\begin{Rem}
\label{rem:1}
The author believes that we can prove the following fact: $7$-dimensional manifolds of Theorem \ref{thm:4} or Main Theorem do not admit special generic maps into ${\mathbb{R}}^n$ ($n=1,2,3,4,5,6$) or more generally, ones satisfying the following conditions into this Euclidean space.
\begin{itemize}
\item The index of each singular point is $0$ or $1$.
\item Each connected component of each preimage contains at most $1$ singular point.
\item Preimages of regular values are always disjoint unions of standard spheres. 
\end{itemize}
Note also that fold maps in Theorem \ref{thm:2} and Theorem \ref{thm:5}, which is in the next section, satisfy these properties.
\end{Rem}
\section{Appendices.}
\label{sec:3}
\subsection{Fold maps in Theorem \ref{thm:2} and ones on manifolds represented as connected sums of total spaces of smooth bundles over the standard sphere of a fixed dimension whose fibers are standard spheres.}
\label{subsec:3.1}
Note also that fold maps in Theorem \ref{thm:2} have been obtained as specific examples of fold maps on manifolds represented as connected sums of total spaces of smooth bundles over the standard sphere of a fixed dimension whose fibers are standard spheres. See \cite{kitazawa, kitazawa2, kitazawa5} for example. Related to this fact, manifolds represented as connected sums of manifolds represented as products of standard spheres admit special generic maps into suitable Euclidean spaces whose dimensions are smaller or equal to the dimensions of the manifolds in considerable cases:
 for example a manifold represented as a connected sum of finitely many copies of $S^{n-1} \times S^{m-n+1}$ admits a special generic map into ${\mathbb{R}}^n$.
  We can construct the map so that the restriction to the singular set is an embedding and that the image is a manifold represented as a boundary connected sum of finitely many copies of $S^{n-1} \times D^1$.

\subsection{Further examples of fold maps on $7$-dimensional closed and simply-connected manifolds.}
\label{subsec:3.2}
We present the following recent result, presenting explicit fold maps on infinitely many $7$-dimensional closed and simply-connected manifolds. Their integral cohomology rings are mutually non-isomorphic. These manifolds are not represented as connected sums of manifolds represented as products of standard spheres as before. We can know this from the structures of the cohomology rings. 
\begin{Thm}[\cite{kitazawa6, kitazawa7}]
\label{thm:5}
Let $\{G_j\}_{j=0}^7$ be a sequence of free and finitely generated commutative groups such that $G_0=G_7=\mathbb{Z}$, that $G_j$ is zero for $j=1,6$, that $G_j$ and $G_{7-j}$ are mutually isomorphic for $j=2,3$ and that the rank of $G_2$ is smaller than or equal to that of $G_4$. 

Then we have the following three.

\begin{enumerate}
\item There exist a $7$-dimensional closed and simply-connected spin manifold $M$ such that $H_j(M;\mathbb{Z})$ is isomorphic to $G_j$ and a fold map $f:M \rightarrow {\mathbb{R}}^4$ satisfying the following properties.
\begin{enumerate}
\item The index of each singular point is always $0$ or $1$.
\item For each connected component of the regular value set of $f$, the preimage of a regular value in each connected component is, empty, diffeomorphic to $S^3$, diffeomorphic to $S^3 \sqcup S^3$, or diffeomorphic to $S^3 \sqcup S^3 \sqcup S^3$.
\end{enumerate}
\item
Furthermore, if $G_2$ and $G_4$ are non-trivial groups, then there exists a family $\{M_{\lambda}\}_{\lambda \in \Lambda}$ of infinitely many $7$-dimensional closed and simply-connected manifolds admitting the fold maps as before such that $H_j(M_{\lambda};\mathbb{Z})$ is isomorphic to $G_j$ and that the integral cohomology rings of $M_{{\lambda}_1}$ and $M_{{\lambda}_2}$ are not isomorphic for distinct ${\lambda}_1,{\lambda}_2 \in \Lambda$.
\item 
Furthermore, the integral cohomology rings of the obtained manifolds cannot be isomorphic to those of Theorem \ref{thm:4}{\rm :} the square of every 2nd integral cohomology class must be divisible by 2.
\end{enumerate}
\end{Thm}
In very specific cases, these manifolds are obtained as ones represented as connected sums of manifolds represented as products of two standard spheres as before.

As Theorem \ref{thm:4}, this result is also seen as one capturing the integral cohomology rings of $7$-dimensional closed and simply-connected manifolds of suitable families via explicit fold maps on them into ${\mathbb{R}}^4$.
\subsection{Remarks on the existence of fold maps.}
\label{subsec:3.3}
Some theory of Eliashberg \cite{eliashberg,eliashberg2} and
related theory of Ando \cite{ando} tell us that from the tangent bundle of a closed manifold we can know whether the manifold admits a fold map into a Euclidean space. For a closed and orientable manifold $M$,
if the Whitney sum of the tangent bundle and a trivial linear bundle over $M$ whose fiber is $\mathbb{R}$ is trivial, then $M$ admits a fold map into Euclidean spaces whose dimensions are smaller than or equal to the dimension of $M$.
For example, a closed and connected manifold obtained by taking connected sums and products one after another starting from (a family of finitely many) standard spheres satisfies the condition. Subsections \ref{subsec:3.1} and \ref{subsec:3.2} say that such manifolds and manifolds of some wider classes have very explicit fold maps in considerable cases.
 Knowing the existence of explicit fold maps and their construction is essentially different from knowing the existence of fold maps only.

It has been difficult to know the existence of fold maps for more general manifolds. As an explicit case, studies concerning the existence of fold maps on real and complex projective spaces have been difficult. For several answers, see \cite{ohmotosaekisakuma, sadykovsaekisakuma, saeki4, saekisakuma}. For example, from these studies, we can know that ${\mathbb{C}P}^2$ does not admit fold maps into ${\mathbb{R}}^n$ for $n=2,3,4$. Remember also pioneering theory \cite{thom, whitney} on generic maps into the plane on closed manifolds whose dimensions are greater than or equal to $2$. 

${\mathbb{C}P}^2$ can be embedded into ${\mathbb{R}}^8$ and by considering a normal bundle or a closed tubular neighborhood, we have a linear bundle over ${\mathbb{C}P}^2$ whose fiber is $S^3$. The total space is an embedded submanifold of ${\mathbb{R}}^8$. By a fundamental argument, it is a $7$-dimensional closed, simply-connected and spin manifold and its integral cohomology ring is shown to be isomorphic to that of ${\mathbb{C}P}^2 \times S^3$. We can also know that its tangent bundle satisfies the previous condition and that it admits a fold map into ${\mathbb{R}}^n$ for $n=1,2,3,4,5,6,7$.
Furthermore, our Main Theorem tells the existence of a very explicit fold map into ${\mathbb{R}}^4$ on the $7$-dimensional manifold.

We present another explicit remark. The 2-dimensional real projective space admits no fold maps into ${\mathbb{R}}^2$ by virtue of \cite{thom,whitney}. We have a $3$-dimensional closed and orientable manifold as the total space of a smooth bundle over the projective space whose fiber is diffeomorphic to $S^1$. This is a so-called {\it graph manifold}. Graph manifolds form an important class of $3$-dimensional closed and orientable manifolds including (orientable) total spaces of smooth bundles over closed surfaces whose fibers are circles and the classes of so-called {\it Lens spaces} and {\it Seifert manifolds} for example. Every graph manifold admits a fold map into ${\mathbb{R}}^2$ such that the restriction to the singular set is an embedding by \cite{saeki3}. \cite{kitazawasaeki} announces that this admits a round fold map into ${\mathbb{R}}^2$. Main Theorem (together with Theorem \ref{thm:1}) may be regarded as a complexified version of this fact.
\section{Acknowledgment.}
\label{sec:4}
\thanks{This work was supported by "The Sasakawa Scientific Research Grant" (2020-2002 : https://www.jss.or.jp/ikusei/sasakawa/). The author is a member of JSPS KAKENHI Grant Number JP17H06128 "Innovative research of geometric topology and singularities of differentiable mappings" (Principal Investigator: Osamu Saeki).}


The author declares that all data essentially supporting the present study are in the present paper. 

\end{document}